\documentclass[12pt, reqno]{amsart}

\usepackage{amsfonts,amssymb}
\usepackage{graphicx}
\usepackage{epsfig}
\usepackage{latexsym}
\usepackage{amsmath,amsthm}
\usepackage{mathrsfs}
\usepackage[all,cmtip]{xy}
\usepackage{comment}
\usepackage[dvipsnames]{xcolor}

\usepackage[colorlinks]{hyperref}
\hypersetup{colorlinks,linkcolor={red},citecolor={blue},urlcolor={red}}
\theoremstyle{plain}
\newtheorem{theorem}{Theorem}[section]
\newtheorem{lemma}[theorem]{Lemma}
\newtheorem{definition}[theorem]{Definition}

\newtheorem{cor}[theorem]{Corollary}
\newtheorem{remark}[theorem]{Remark}

\numberwithin{equation}{section}

\begin{document}
\baselineskip=15.5pt

\title{Semi--finite Vector Bundles on Complex Tori}
\author[P. Adroja]{Pavan~Adroja}
\address{Department of Mathematics, IIT Gandhinagar,
 Near Village Palaj, Gandhinagar - 382055, India}
 \email{pavan.a@alumni.iitgn.ac.in, adrojapavan@gmail.com}
\author[S. Amrutiya]{Sanjay~Amrutiya}
\address{Department of Mathematics, IIT Gandhinagar,
 Near Village Palaj, Gandhinagar - 382055, India}
 \email{samrutiya@iitgn.ac.in}
\date{}

\begin{abstract}
We study finite and semi--finite vector bundles on complex tori. We give an explicit decomposition of such bundles in terms of torsion and unipotent factors. As a consequence, we prove that the extended Nori fundamental group scheme of a complex torus decomposes as the product of its \'etale fundamental group scheme and its unipotent fundamental group scheme.
\end{abstract}

\maketitle
\section{Introduction}

The notion of finite vector bundles was studied by Nori in \cite{No}. When $k$ is a field of characteristic zero and $X$ is a proper integral scheme having a rational point $x\in X(k)$, the category of finite vector bundles, equipped with the fiber functor $x^*$ induced by pullback, forms a neutral Tannakian category. The \emph{Nori fundamental group scheme} $\pi^{\rm N}(X,x)$ of $(X,x)$ is a Tannakian fundamental group scheme corresponding to the category of finite bundles. This is a natural strategy, initiated by Nori, to study various classes of vector bundles that form Tannakian categories and the associated fundamental group schemes. Since $k$ has zero characteristic, it is well known that the Nori fundamental group scheme $\pi^{\rm N}(X,x)$ is isomorphic to the \'etale fundamental group scheme $\pi^{\rm \acute{e}t}(X,x)$.

Otabe \cite{Ot} examined a category that is larger than the category of finite bundles in characteristic zero. He introduced the concept of semi--finite bundles, those filtered by finite bundles, and demonstrated that the category of semi--finite bundles over $X$ forms a neutral Tannakian category. The corresponding affine group scheme is referred to as an \emph{extended Nori fundamental group scheme}, denoted by $\pi^{\mathrm{EN}}(X, x)$.

In this article, we study finite and semi--finite holomorphic vector bundles on complex tori. We explicitly describe the decomposition of finite and semi--finite vector bundles on complex tori (see Lemma \ref{finite is sum of torsion on torus} and Lemma \ref{sf=f+u}). Furthermore, employing the Deligne tensor product of Tannakian categories, we prove that the extended Nori fundamental group scheme of a complex torus decomposes as the product of the \'etale fundamental group scheme and the unipotent fundamental group scheme. An analogous result was established in \cite{AA} for abelian varieties using a different method.

\section{Preliminaries}\label{section: Preliminaries}
Let $X$ be a compact complex manifold. 
Let $f(t) = \sum_{k=0}^n a_k t^k\,\in\, \mathbb{N}[t]$ be a polynomial
whose coefficients are nonnegative integers. For any holomorphic vector bundle
$E$ over $X$, define the holomorphic vector bundle
$$
f(E)\ :=\ \bigoplus_{k=0}^n (E^{\otimes k})^{\oplus a_k},
$$
where $E^{\otimes 0}$ denotes the trivial line bundle ${\mathcal O}_X$.

\begin{definition}[{\cite{No}}]\label{def-finite}\rm{
A holomorphic vector bundle $E$ over $X$ is said to be \textit{finite} if there are two distinct polynomials $f$ and $g$,
whose coefficients are nonnegative integers, such that $f(E)$ is holomorphically isomorphic to $g(E)$.
}
\end{definition} 

The arguments of \cite[Chapter I, Lemma 3.1]{No} can be adapted to the setting of compact complex manifolds, yielding equivalent criteria for a bundle to be finite. In particular, a line bundle is finite if and only if it is torsion. Recall that a line bundle $L$ on $X$ is called \emph{torsion} if there exists a positive integer $m$ such that  $L^{\otimes m} \simeq \mathcal{O}_X$.

\begin{definition}[{\cite{No}}]\label{def-unipotent}\rm{
A holomorphic vector bundle $E$ on $X$ is called \textit{unipotent} if it has a filtration of holomorphic subbundles 
$$
E\,=\,E_0\,\supset\, E_1 \,\supset\, E_2 \,\supset\,\cdots\, \supset\, E_{n+1}\,=\,0
$$ such that $E_i/E_{i+1}$ is a trivial bundle, for all $0\, \leq\, i\, \leq\, n$.
}
\end{definition}

\begin{definition}[{\cite{Ot}}]\label{def-semifinite}\rm{
A holomorphic vector bundle $E$ on $X$ is called \textit{semi--finite} if it has a filtration of holomorphic subbundles 
$$
E\,=\,E_0\,\supset\, E_1 \,\supset\, E_2 \,\supset\,\cdots\, \supset\, E_{n+1}\,=\,0
$$ such that $E_i/E_{i+1}$ is a finite bundle, for all $0\, \leq\, i\, \leq\, n$.
}
\end{definition}

Let \(\mathcal{C}^{\mathrm{N}}(X)\), \(\mathcal{C}^{\mathrm{uni}}(X)\), and \(\mathcal{C}^{\mathrm{EN}}(X)\) denote the full subcategories of the category of holomorphic vector bundles on \(X\), whose objects are finite, unipotent, and semi--finite bundles, respectively. Let \(\star\) be a placeholder for any of the symbols \(\mathrm{N}, \mathrm{EN}, \mathrm{uni}\).  

Fix a base point \(x_0 \in X\). Then, there exists a neutral fibre functor  
\[
{x_0}^* \colon \mathcal{C}^{\star}(X) \longrightarrow \mathbf{Vect}_{\mathbb{C}}
\]
that maps each vector bundle \(E\) to its fiber \(E_{x_0}\) at \(x_0\), where $\mathbf{Vect}_{\mathbb{C}}$ denotes the category of finite-dimensional $\mathbb{C}$-vector spaces. The pair \((\mathcal{C}^{\star}(X), {x_0}^*)\) forms a neutral Tannakian category. By Tannakian duality \cite[Theorem 1.12]{PD}, there exists an affine group scheme \(\pi^{\star}(X, x_0)\) over \(\mathbb{C}\) such that the category of finite-dimensional representations \(\mathrm{Rep}(\pi^{\star}(X, x_0))\) is equivalent to the category \(\mathcal{C}^{\star}(X)\).  

The affine group schemes \(\pi^{\mathrm{N}}(X, x_0)\), \(\pi^{\mathrm{uni}}(X, x_0)\), and \(\pi^{\mathrm{EN}}(X, x_0)\) are known as the \emph{Nori fundamental group scheme}, the \emph{unipotent fundamental group scheme}, and an \emph{extended Nori fundamental group scheme}, respectively. For further details regarding these affine group schemes, see \cite{No, Ot, AA}. Since the field of complex numbers has characteristic zero, and every finite group scheme over a field of characteristic zero is \'etale, the Nori fundamental group scheme \(\pi^{\mathrm{N}}(X, x_0)\) is isomorphic to the \'etale fundamental group scheme \(\pi^{\mathrm{\acute{e}t}}(X, x_0)\).

\begin{remark}
\rm	In \cite{BV}, Borne and Vistoli introduced the virtually unipotent fundamental group by adapting Nori’s construction in \cite[Chapter-II]{No}. They replace finite group scheme with virtually unipotent, i.e., an extension of finite group scheme by a unipotent. The Tannakian category associated to the virtually unipotent fundamental group is precisely the category of semi--finite vector bundles (see \cite[Theorem 10.5]{BV}).
\end{remark}

\subsection*{Tensor product of categories}
In \cite[\S 5, p. 142]{PD}, Deligne defined a tensor product of abelian categories. Let $\mathcal{T}_1$ and $\mathcal{T}_2$ be two $k$-linear abelian categories. The Deligne tensor product of $\mathcal{T}_1$ and $\mathcal{T}_2$ is a $k$-linear abelian category $\mathcal{T}$ with $k$-bilinear and right-exact-in-each-variable functor $\otimes_D:\mathcal{T}_1\times \mathcal{T}_2 \rightarrow \mathcal{T}$ which satisfies the universal property: for any $k$-linear abelian category $\mathcal{C}$ with $k$-bilinear and right-exact-in-each-variable functor $F:\mathcal{T}_1\times \mathcal{T}_2 \rightarrow \mathcal{C}$, there exists a $k$-linear right exact functor $F':\mathcal{T} \rightarrow \mathcal{C}$ such that the diagram 
\[
\xymatrix{
	\mathcal{T}_1\times \mathcal{T}_2 \ar[d]_{F}  \ar[rr]^{\otimes_D}  & &  \mathcal{T} \ar[dll]^{F'} \\
	\mathcal{C}  & &  \\
}
\]
commutes. 
We will denote $\mathcal{T}$ by $\mathcal{T}_1\otimes_D \mathcal{T}_2$. 

\begin{remark}\label{Rep(G_1*G_2)= Rep(G_1)*Rep(G_2)}
	\rm By \cite[5.18, 6.21]{PD}, one can conclude that, for affine group schemes $G_1$ and $G_2$ over a field $k$, the category $\mathrm{Rep}(G_1\times_k G_2)$ is equivalent to the Deligne tensor product $\mathrm{Rep}(G_1)\otimes_D \mathrm{Rep}(G_2)$ of categories $\mathrm{Rep}(G_1)$ and $\mathrm{Rep}(G_2)$.
\end{remark}

\section{Semi--finite bundles on complex torus}
In this section, we will see that an extended Nori fundamental group scheme of a complex torus is isomorphic to the product of the Nori fundamental group scheme and the unipotent fundamental group scheme. Throughout this section, let $X=V/\Lambda$ be a complex torus of complex dimension $g$. Then, we have $\Lambda \simeq \mathbb{Z}^{2g}$.

\begin{lemma}\label{finite is sum of torsion on torus}
	Let $X=V/\Lambda$ be a complex torus of complex dimension $g$, and let $E$ be a finite holomorphic vector bundle on $X$. Then, 
	$$E\simeq \bigoplus_{i=1}^n L_i\,,$$
	where $L_i$'s are torsion line bundles on $X$.
\end{lemma}
\begin{proof}
Since $\pi^{\rm N}(X, 0)=\pi^{\rm \acute{e}t}(X, 0) =\widehat{\Lambda}=\widehat{\mathbb{Z}}^{2g}$, every indecomposable representation of $\pi^{\rm N}(X, 0)$ has rank $1$. Thus, every indecomposable finite holomorphic vector bundle has rank $1$.
\end{proof}

\begin{lemma}\label{n_*O_X}
Let $X$ be a complex torus of complex dimension $g$, and let $[n]:X\rightarrow X$ denote the multiplication-by-$n$ morphism. Then, the pushforward $[n]_*\mathcal{O}_X$ decomposes as 
$$
	[n]_{*}\mathcal{O}_X \simeq \bigoplus_{i} L_i ,
$$
where the $L_i$ run over all $n$-torsion line bundles on $X$.
\end{lemma}
\begin{proof}
Since $[n]$ is a finite \'etale morphism of degree $n^{2g}$, the pushforward $[n]_{*}\mathcal{O}_X$ is a vector bundle of rank $n^{2g}$ and it has the natural action of the finite abelian group $G = X[n] \simeq (\mathbb{Z}/n\mathbb{Z})^{2g}$ of $n$-torsion points; where $X[n]:= \mathrm{Ker}([n])$. In other words, $[n]_{*}\mathcal{O}_X$ is a $G$-equivariant vector bundle. Since $G$ is finite abelian, the representation of $G$ splits as a direct sum of one-dimensional characters and hence $[n]_{*}\mathcal{O}_X$ decomposes as a direct sum of line bundles equipped with $G$-action. These line bundles are precisely the $n$-torsion line bundles on $X$. Since $[n]^{*}L_i \simeq L_i^{\otimes n} \simeq \mathcal{O}_X$, by the adjunction property, we have a non-trivial morphism $L_i \rightarrow [n]_{*}\mathcal{O}_X$. This morphism splits because the bundle $[n]_*\mathcal{O}_X$ is semisimple. As there are exactly $n^{2g}$ distinct $n$-torsion line bundles and $[n]_{*}\mathcal{O}_X$ has rank $n^{2g}$, each torsion line bundle appears exactly once in the decomposition.
\end{proof}

\begin{lemma}\label{H^0=H^1=0}
	Let $X$ be a complex torus of complex dimension $g$, and let $L$ be a non-trivial torsion line bundle on $X$. Then, 
	$$\mathrm{H}^0(X,L)=\mathrm{H}^1(X,L)=0.$$
\end{lemma}
\begin{proof}
	Since $L$ is a torsion line bundle, the degree of $L$ is zero. This implies that a non-zero global section of $L$ cannot vanish anywhere. So, if such a section exists, then it gives an isomorphism $L \simeq \mathcal{O}_X$. Hence, $\mathrm{H}^0(X,L)=0$.
	
	Let $n > 1$ be such that $L^{\otimes n} \simeq \mathcal{O}_X$. Consider the multiplication-by-$n$ morphism $ [n] : X \rightarrow X $. By Lemma \ref{n_*O_X}, the pushforward $[n]_{*}\mathcal{O}_X$ decomposes as
	$$
	[n]_{*}\mathcal{O}_X \simeq \bigoplus_{i} L_i ,
	$$
	where the $L_i$ run over all $n$-torsion line bundles on $X$. In particular, $\mathrm{H}^{1}(X,L_i)$ is a direct summand of $\mathrm{H}^{1}(X, [n]_{*}\mathcal{O}_X)$. By the Leray spectral sequence, we have an inclusion 
	$$0 \longrightarrow \mathrm{H}^{1}(X, [n]_{*}\mathcal{O}_X) \longrightarrow \mathrm{H}^{1}(X, \mathcal{O}_X),$$
	which is isomorphism as $\mathrm{H}^{1}(X,\mathcal{O}_X)$ is a direct summand of $\mathrm{H}^{1}(X, [n]_{*}\mathcal{O}_X)$. Therefore, $\mathrm{H}^{1}(X, L_i) = 0$ for every $L_i \neq \mathcal{O}_X$.
\end{proof}

\begin{lemma}\label{sf=f+u}
	Let $X$ be a complex torus of complex dimension $g$, and let $E$ be a semi--finite holomorphic vector bundle on $X$. Then, 
	$$E\simeq \bigoplus_{i=1}^n U_i \otimes L_i\,, $$
	where $L_i$'s are torsion line bundles on $X$ and $U_i$'s are unipotent bundles on $X$.
\end{lemma}
\begin{proof}
	We proceed by induction on the rank of \(E\). If \(\mathrm{rank}(E)=1\), then by the definition of semi--finite bundles and Lemma~\ref{finite is sum of torsion on torus}, \(E\) must be a torsion line bundle. In this case, the statement holds with \(U=\mathcal{O}_X\) and \(L=E\). Assume the statement holds for all semi--finite bundles of rank \(< r\), and let \(\mathrm{rank}(E)=r\). By Lemma~\ref{finite is sum of torsion on torus} and the definition of semi--finite bundles, there exists a filtration
	\[
	0 = E_0 \subset E_1 \subset \cdots \subset E_{r-1} \subset E_r = E
	\]
	by subbundles such that each successive quotient \(E_i/E_{i-1}\) is a torsion line bundle. In particular, we have a short exact sequence
	\begin{equation}\label{eq:main-sequence}
		0 \longrightarrow E_{r-1} \longrightarrow E \longrightarrow L \longrightarrow 0 \,,
	\end{equation}
	where \(L = E_r/E_{r-1}\) is a torsion line bundle. Since \(\mathrm{rank}(E_{r-1})=r-1\), by the induction hypothesis, we have
	\[
	E_{r-1} \simeq \bigoplus_i U_i \otimes L_i \,,
	\]
	where each \(U_i\) is unipotent and each \(L_i\) is a torsion line bundle on \(X\). Split the direct sum for $E_{r-1}$ according to whether $L_j\simeq L$ or $L_j \not \simeq L$. Write
	$$E_{r-1} \simeq W \oplus (U_0 \otimes L) \,,$$
	where $W:= \bigoplus_{j:L_j \not \simeq L} U_j \otimes L_j$ and $U_0:= \bigoplus_{j:L_j \simeq L} U_j$. Note that $U_0$ is unipotent as it is direct sum of unipotent bundles.

	Since 
	$$\mathrm{Ext}^1 \bigl( L, E_{r-1} \bigr) \simeq \mathrm{Ext}^1 \bigl( L, W \bigr) \oplus \mathrm{Ext}^1 \bigl( L, U_0 \otimes L \bigr)\,,$$
	we have $E \simeq W' \oplus E'$ such that $W'$ and $E'$ fit into the following exact sequences: 
	\begin{equation}\label{E:W'}
		0\longrightarrow W \longrightarrow W' \longrightarrow L \longrightarrow 0 \,.
	\end{equation}
	\begin{equation}\label{E:E'}
		0\longrightarrow U_0 \otimes L \longrightarrow E' \longrightarrow L \longrightarrow 0 \,.
	\end{equation}
	Tensoring the exact sequence \eqref{E:E'} with \(L^{-1}\), and using that tensoring by a line bundle is an exact auto-equivalence of the category of vector bundles on \(X\), we obtain an equivalent extension
	\begin{equation}\label{eq:quence}
		0 \longrightarrow U_0 \longrightarrow E'\otimes L^{-1} \longrightarrow \mathcal{O}_X \longrightarrow 0 \,.
	\end{equation}
	From the above sequence, $U':=E'\otimes L^{-1}$ is a unipotent bundle. Hence, $E'=U' \otimes L$. Therefore, it suffices to show that the exact sequence \eqref{E:W'} splits. The sequence \eqref{E:W'} splits if the sequence 
	$$0 \longrightarrow U_j \otimes L_j \longrightarrow F_j \longrightarrow L \longrightarrow 0 $$
	splits for every $j$, where $L_j \not \simeq L$. Thus it is enough to prove that any extension
	\begin{equation}\label{eq:Ui-extension}
		0 \longrightarrow U \otimes L' \longrightarrow F \longrightarrow \mathcal{O}_X \longrightarrow 0
	\end{equation}
	splits, where \(U\) is a unipotent bundle and \(L'\) is a non-trivial torsion line bundle on \(X\).
	
	Since \(U\) is unipotent, it admits a filtration
	\[
	0 = U_0 \subset U_1 \subset \cdots \subset U_{m-1} \subset U_m = U
	\]
	with successive quotients \(U_j/U_{j-1} \simeq \mathcal{O}_X\). Tensoring by \(L'\) gives
	\[
	0 \longrightarrow U_{m-1}\otimes L' \longrightarrow U \otimes L' \longrightarrow L' \longrightarrow 0 \,.
	\]
	By induction on \(\mathrm{rank}(U)\), it suffices to show that any extension
	\begin{equation}\label{eq:L-extension}
		0 \longrightarrow L' \longrightarrow G \longrightarrow \mathcal{O}_X \longrightarrow 0
	\end{equation}
	splits. Since $\mathrm{Ext}^1(\mathcal{O}_X, L') \simeq \mathrm{H}^1(X, L')$ and $\mathrm{H}^1(X, L')=0$ by Lemma \ref{H^0=H^1=0}, the sequence (\ref{eq:L-extension}) splits.
\end{proof}

\begin{remark}
\rm	Lemma \ref{sf=f+u} is closely related to classical results on homogeneous vector bundles on complex tori. A holomorphic vector bundle $E$ on a complex torus $X$ is called \emph{homogenenous} if it is translation invariant. These bundles are iterated extensions of algebraically trivial line bundles (see \cite{Mat59, Mor59}). Hence, by Lemma \ref{finite is sum of torsion on torus}, semi--finite vector bundles are homogeneous.
\end{remark}

\begin{theorem}\label{Prod EN=N*uni on torus}
Let $X$ be a complex torus of complex dimension $g$. Then, we have 
\begin{align*}
	\pi^{\rm EN}(X,0) & \simeq \pi^{\mathrm{\acute{e}t}}(X,0) \times \pi^{\rm uni}(X,0).
\end{align*}
\end{theorem}
\begin{proof}
Consider the Deligne tensor product of the categories $\mathcal{C}^{\rm N}(X)$ and $\mathcal{C}^{\rm uni}(X)$. That means, there is a $k$-bilinear and right-exact-in-each-variable functor 
\begin{align}\label{D-tensor of N and uni}
\otimes_D: \mathcal{C}^{\rm N}(X) \times \mathcal{C}^{\rm uni}(X) \longrightarrow \mathcal{C}^{\rm N}(X) \otimes_{D} \mathcal{C}^{\rm uni}(X)
\end{align}
satisfying the universal property. Let $T: \mathcal{C}^{\rm N}(X) \times \mathcal{C}^{\rm uni}(X) \rightarrow \mathcal{C}^{\rm EN}(X)$ be a functor defined by $(F, U)\mapsto F\otimes U$ for $F\in \mathrm{Ob}(\mathcal{C}^{\rm N}(X))$ and $U\in \mathrm{Ob}(\mathcal{C}^{\rm uni}(X))$. Clearly, it is a $k$-bilinear and exact-in-each-variable functor. Hence, there exists a $k$-linear right exact functor $T': \mathcal{C}^{\rm N}(X) \otimes_{D} \mathcal{C}^{\rm uni}(X) \rightarrow \mathcal{C}^{\rm EN}(X)$ which makes the following diagram commutative.
\begin{equation}\label{eq:DT}
\xymatrix{
\mathcal{C}^{\rm N}(X) \times \mathcal{C}^{\rm uni}(X) \ar[d]_{T}  \ar[r]^{\quad \quad \quad \otimes_D}  &  **[r] \mathcal{C}^{\rm N}(X) \otimes_{D} \mathcal{C}^{\rm uni}(X) \ar[dl]^{T'} \\
\mathcal{C}^{\rm EN}(X)  &   \\
}
\end{equation}
Let $E$ be a semi--finite bundle on $X$. Then, by Lemma \ref{sf=f+u}, we have 
$$
E\simeq \bigoplus_{i=1}^n L_i \otimes U_i \;,
$$ 
where $L_i \in \mathrm{Ob}(\mathcal{C}^{\rm N}(X))$ and $U_i \in \mathrm{Ob}(\mathcal{C}^{\rm uni}(X))$. 
Since $T'$ is right exact, using the commutativity of the diagram \eqref{eq:DT}, we have
$$
T' \Big(\bigoplus_{i} \otimes_D(L_i, U_i)\Big) = 
\bigoplus_{i} \Big(T' \big( \otimes_D(L_i, U_i)\big)\Big) \simeq E\;.
$$ 
Hence, the functor $T'$ is essentially surjective. To see that the functor $T'$ is fully faithful, let $L_1,L_2 \in \mathrm{Ob}(\mathcal{C}^{\rm N}(X))$ and $U_1, U_2 \in \mathrm{Ob}(\mathcal{C}^{\rm uni}(X))$. Since the functor $\otimes_D$ is universal and $k$-bilinear, we have $\mathrm{Hom}(L_1,L_2) \otimes_k \mathrm{Hom}(U_1,U_2)\simeq \mathrm{Hom}(\otimes_D(L_1,U_1), \otimes_D(L_2,U_2)) $. Note that 
\begin{align*}
\mathrm{Hom}\left(T'\left(\otimes_D(L_1,U_1)\right), T'\left(\otimes_D(L_2,U_2)\right)\right) & = \mathrm{Hom}(L_1 \otimes U_1,L_2 \otimes U_2) \\
	& \simeq \mathrm{Hom}(L_1,L_2) \otimes_k \mathrm{Hom}(U_1,U_2).
	\end{align*}
Hence, the functor $T'$ is fully faithful. The assertion follows by Remark \ref{Rep(G_1*G_2)= Rep(G_1)*Rep(G_2)}.
\end{proof}

\begin{remark}
\rm Let $A$ be an abelian variety defined over an algebraically closed field $k$ of characteristic zero. According to \cite[Corollary 3.5]{AA}, we have 
$$\pi^{\rm EN}(A,0)\simeq \pi^{\rm \acute{e}t}(A,0)\times_k \pi^{\rm uni}(A,0).$$
This result provides alternative proof of Theorem \ref{Prod EN=N*uni on torus} when $X$ is an algebraic torus. It is important to note that the proof of Theorem~\ref{Prod EN=N*uni on torus} does not require the affine group scheme $\pi^{\rm EN}(X,0)$ to be abelian. In contrast, the proof of \cite[Corollary 3.5]{AA} requires that the affine group scheme $\pi^{\rm EN}(X,0)$ be abelian.
\end{remark}

\begin{remark}\label{R:ProdF U}
	\rm
	In \cite[Chapter~IV, Lemma~8]{No}, Nori established the product formula for the unipotent fundamental group scheme in the algebraic context. The same argument can also be applied in the complex analytic setting, since the K\"unneth formula holds for compact complex manifolds as well. 
Consequently, if $X$ is a complex torus, then the unipotent fundamental group scheme $\pi^{\rm uni}(X,0)$ is an abelian group scheme.
\end{remark}

\begin{cor}
	Let $X$ be a complex torus of complex dimension $g$. Then, we have 
	$$\pi^{\rm EN}(X,0) \simeq \widehat{\mathbb{Z}}^{2g} \times (\mathbb{G}_{a, \mathbb{C}})^{g},$$
	where $\mathbb{G}_{a, \mathbb{C}}$ denote the additive group scheme over the field of complex numbers.
\end{cor}
\begin{proof}
By Remark \ref{R:ProdF U}, the unipotent fundamental group scheme $\pi^{\rm uni}(X,0)$ is an abelian group scheme. Hence, the tangent space of $\pi^{\rm uni}(X,0)$ at the identity is dual to $\mathrm{Ext}^1(\mathcal{O}_X, \mathcal{O}_X)$, and
$$\mathrm{Lie}(\pi^{\rm uni}(X,0))\simeq \mathrm{Ext}^1(\mathcal{O}_X, \mathcal{O}_X)^{\vee}\simeq \mathrm{H}^1(X,\mathcal{O}_X)^{\vee}\simeq \mathbb{C}^{g}$$
as $\mathrm{dim}(\mathrm{H}^1(X,\mathcal{O}_X))=g$. In characteristic zero, a unipotent group scheme is completely determined by its Lie algebra. Thus, an affine group scheme $\pi^{\rm uni}(X,0)$ is given by the functor
$$\pi^{\rm uni}(X,0): \mathbf{Alg}_{\mathbb{C}} \longrightarrow \mathbf{Grp}$$
which maps $R$ to $\mathrm{Hom}_{\mathbb{C}\mathrm{-lin}}(\mathbb{C}^g, R)$, where $\mathbf{Alg}_{\mathbb{C}}$ denotes the category of $\mathbb{C}$-algebra and $\mathbf{Grp}$ denotes the category of groups. Note that $\mathrm{Hom}_{\mathbb{C}\mathrm{-lin}}(\mathbb{C}^g, R) \simeq \mathrm{Hom}_{\mathbb{C}\mathrm{-alg}}(\mathbb{C}[x_1,\ldots, x_g], R)$ for any $R \in \mathbf{Alg}_{\mathbb{C}}$. Hence, $\pi^{\rm uni}(X,0)$ is isomorphic to $(\mathbb{G}_{a, \mathbb{C}})^{g}$.
\end{proof}

\section*{Acknowledgements}
The authors would like to thank the anonymous referee for constructive comments and valuable suggestions that improved this manuscript. The research work of Pavan Adroja and Sanjay Amrutiya is financially supported by the Science and Engineering Research Board-Department of Science and Technology (SERB-DST), Government of India, under Project No. CRG/2023/000477.



\begin{thebibliography}{012345}
\bibitem{AA} P.~Adroja and S.~Amrutiya, \emph{On an extension of Nori and local fundamental group schemes}, Comm. Algebra \textbf{53}(10) (2025), 4241-4255.

\bibitem{BV} N. Borne and A. Vistoli, \emph{Fundamental gerbes}, Algebra Number Theory {\bf 13} (2019), no.~3, 531--576.

\bibitem{PD} P.~Deligne, \emph{Catégories tannakiennes}. in The Grothendieck Festschrift, Vol. II, Progr. Math., vol. \textbf{87}, Birkhäuser, Boston, MA, 1990, p. 111-195.

\bibitem{Mat59} Y.~Matsushima, \emph{Fibr\'es holomorphes sur un tore complexe}, Nagoya Math. J. {\bf 14} (1959), 1--24.

\bibitem{Mor59} A. Morimoto, \emph{Sur la classification des espaces fibr\'es vectoriels holomorphes sur un tore complexe admettant des connexions holomorphes}, Nagoya Math. J. {\bf 15} (1959), 83--154.

\bibitem{No} M.~V.~Nori, \emph{The fundamental group-scheme}, Proc. Indian Acad. Sci. Math. Sci. \textbf{91} (1982), no. 2, p. 73-122. 
		
\bibitem{Ot} S.~Otabe, \emph{An extension of Nori fundamental group}, Comm. Algebra \textbf{45} (2017), 3422-3448.



\end{thebibliography}
\end{document}